\documentclass[12pt, a4paper]{amsart}

\usepackage{amsfonts}
\usepackage{amssymb}
\usepackage{amsmath, amsthm,color}
\usepackage{pdfsync}
\usepackage{bbm, dsfont}

\newcommand{\pd}{\partial}

\newcommand{\eps}{\varepsilon}

\newcommand{\om}{\omega}
\newcommand{\Om}{\Omega}

\newcommand{\bC}{\mathbb{C}}

\newcommand{\bT}{\mathbb{T}}
\newcommand{\bD}{\mathbb{D}}

\newcommand{\cB}{\mathcal{B}}

\newtheorem{theorem}{Theorem}[section]
\newtheorem{corollary}[theorem]{Corollary}
\newtheorem{lemma}[theorem]{Lemma}
\newtheorem{prop}[theorem]{Proposition}

\newtheorem{remark}[theorem]{Remark}

\numberwithin{equation}{section}

\begin{document}

\title[Counting eigenvalues of Schr\"odinger operators]{Counting eigenvalues of Schr\"odinger operators with fast decaying complex potentials}

\author{A.~Borichev}
\address[A.~Borichev]{Aix-Marseille University, CNRS, Centrale Marseille, I2M, France}
\email{alexander.borichev@math.cnrs.fr}

\author{R.~L.~Frank}
\address[R.~Frank]{Mathematisches Institut, Ludwig-Maximilians Universit\"at M\"unchen, Germany, and Department of Mathematics, California Institute of Technology, Pasadena, CA 91125, USA}
\email{r.frank@lmu.de}

\author{A.~Volberg}
\address[A.~Volberg]{Department of Mathematics, Michigan State University, East Lansing, MI 48823, USA}
\email{volberg@math.msu.edu}

\thanks{The second and third author acknowledge partial support by the U.S. National Science Foundation through grants DMS-1363432 (R.L.F.) and DMS-1600065 (A.V.). This paper is based upon work supported by the National Science Foundation under Grant No. DMS-1440140 while the second and the third author were in residence at the Mathematical Sciences Research Institute in Berkeley, California, during the Spring 2017 semester. The first and the third authors were partially supported by the  grant  346300 for IMPAN from the
Simons Foundation and the matching 2015-2019 Polish MNiSW fund.}

\makeatletter
\@namedef{subjclassname@2010}{
  \textup{2010} Mathematics Subject Classification}
\makeatother

\begin{abstract} 
We give a sharp estimate of the number of zeros of analytic functions in the unit disc belonging to analytic quasianalytic Carleman--Gevrey classes. As an application, we estimate the number of the eigenvalues for discrete Schr\"odinger operators with rapidly decreasing complex-valued potentials,   and, more generally, for non-symmetric Jacobi matrices.
\end{abstract}

\maketitle 

\section{Introduction and main results}

Bounding the number of eigenvalues of Schr\"odinger-type operators is a classical topic in spectral theory with many applications in mathematical physics. The situation for Schr\"odinger operators with real-valued potentials has been understood for a long time. The qualitative question of whether the operator has finitely or infinitely many eigenvalues depends on whether the potential decays faster or slower than $|x|^{-2}$ at infinity. This qualitative result is accompanied by quantitative upper bounds on the number of eigenvalues like, for instance, the celebrated inequalities by Bargman or by Cwikel--Lieb--Rozenblum. For more details and references we refer to the textbooks \cite{ReSi4,D2}. All these results hold, mutatis mutandis, for discrete Schr\"odinger operators and for Jacobi matrices.

In contrast, the situation for Schr\"odinger operators with \emph{complex-valued} potential is significantly less understood. Such operators are relevant in applications as well, for instance, in the modeling of dissipative phenomena and also as technical tools in the study of resonances of Schr\"odinger operators with real-valued potentials. For further informations, we refer to \cite{Da,EmTr,FrLaLiSe,BoGoKu, HaKa} and references therein.

The conditions for finiteness or infiniteness of the number of eigenvalues in the case of complex-valued potentials are remarkably different from those in the real-valued case. In two fundamental papers \cite{Pa1,Pa2}, Boris Pavlov showed that in the case of complex-valued potentials the number of eigenvalues is finite provided that the potential is bounded by $C_1 e^{-c_2 |x|^{1/2}}$ and that this condition is optimal in the sense that for any $\alpha<1/2$ there is a potential bounded by $C_1’ e^{-c_2’ |x|^{\alpha}}$ with an infinite number of eigenvalues. This is in striking contrast to the real-valued case. Pavlov's result concerns continuous Schr\"odinger operators, but, as pointed out in \cite{GoEg} the result is also true for Jacobi matrices.

This settles the qualitative aspect of the question, but leaves open the question of finding quantitative upper bounds on the number of eigenvalues, for instance, in terms of the constants $C_1$ and $c_2$ in the bound $C_1 e^{-c_2 |x|^{1/2}}$ on the potential. Pavlov's method is intrinsically non-quantitative and cannot provide such a bound. There has been no progress on this question in the past fifty years. 

The fundamental difference between the self-adjoint case of real-valued potentials and the non-selfadjoint case of complex-valued potentials is the lack of a spectral theorem and of a variational characterization of eigenvalues in the latter case. Those play a big role in obtaining both qualitative and quantitative results on eigenvalues in the self-adjoint case. What remains in the non-selfadjoint case are either operator-theoretic tools (as used, for instance, in \cite{FrLaLiSe,GoKu,Ha}) or tools from complex analysis (as used, for instance, in \cite{DeHaKa,BoGoKu,FrSa,Fr3}). The latter typically give more precise results and were also used in Pavlov's original work. The idea is to realize the eigenvalues as zeros of an analytic function (typically a determinant-like quantity), translate bounds on the potential into bounds on this analytic function and then to use complex analytic bounds on the number of zeros in terms of the controlled quantities.

The simplest situation occurs when the potential decays exponentially. In this case, the relevant analytic function has an analytic continuation in a neighborhood of its original domain and bounds on the number of zeros can simply be obtained by Jensen's theorem from complex analysis. This technique was first carried out for complex-potentials by Na{\u\i}mark \cite{Na}. For recent bounds in this case see, for instance, \cite{FrLaSa} and references therein.

In Pavlov's case, where the potential decays like $C_1 e^{-c_2 |x|^{1/2}}$, the relevant analytic function does, in general, not have an analytic continuation to a larger set. To deduce nevertheless that there are only finitely many zeros, Pavlov uses ideas from analytic quasi-analyticity and shows that the function belongs to a Gevrey class and therefore cannot have infinitely many zeros.

In order to obtain a quantitative version of Pavlov's theorem, we therefore need to prove bounds on the number of zeros of functions from a Gevrey class. This is an interesting problem in complex analysis and is, in fact, the main result of this paper. We also show that, at least in an important special case, our bounds are almost sharp.

Combining Pavlov's ideas with  our results on Gevrey class functions we will be able to obtain an explicit bound on the number of eigenvalues in terms of the parameters controlling the size and variation of the potential. We carry this out in the setting of discrete one-dimensional Schr\"odinger operators or Jacobi matrices, since this is technically slightly simpler. In principle, our methods should also work for continuous, multi-dimensional Schr\"odinger operators. They might also be useful in the spectral theory of other non-selfadjoint operators.


\subsection{Smooth functions analytic in the unit disc}

Consider a class of analytic functions in the unit disc $\bD$ which are smooth up to the boundary. If the class is sufficiently 
small, then it satisfies the so called (analytic) quasianalyticity property: any function from the class with infinitely many zeros vanishes identically. 
More precisely, consider the class of functions $f$ analytic in the unit disc such that 
$$
|\hat f(n)|\le e^{-p_n},\qquad n\ge 0,
$$
where
\begin{equation}
\label{eq:taylor}
f(z)=\sum_{n=0}^\infty\hat f(n) z^n,\qquad z\in \bD,
\end{equation}
and $\{p_n\}$ is a sufficiently regular 
sequence. 
Then the condition
\begin{equation}
\label{pnsum}
\sum_{n=0}^\infty \frac{p_n}{1+n^{3/2}} =\infty
\end{equation}
is necessary and sufficient for this class of analytic functions to be quasianalytic in the sense mentioned above, see \cite{Ca} and \cite{Ko}.

Given a function from an analytic quasianalytic class, it is natural to ask for a quantitative bound on the number of zeros. Of course, to get a meaningful answer, 
we have to impose a normalization like
\begin{equation*}
\label{A}
|f(0)| \ge \exp(-A)\,.
\end{equation*}

In this paper, we deal with an important special case of this question concerning analytic quasianalytic Gevrey classes. 

In what follows we denote by $\mathbb D(z,r)$ the disc centered at $z\in\mathbb C$ of radius $r>0$, $\mathbb D(r)=\mathbb D(0,r)$, 
$\bD=\mathbb D(1)$. 
As usual, $m_2$ denotes planar Lebesgue measure.

We fix $\beta_0> 0$ and consider $\beta\in [0,\beta_0]$. (Thus, we are considering arbitrary $\beta>0$. The sole purpose of the parameter $\beta_0$ is to track the dependence of our constants -- in fact, they will typically only depend on $\beta_0$.) We consider the class $\mathfrak A_\beta$ of functions $f$ analytic in the unit disc and smooth up the boundary 
determined by restrictions of their Taylor coefficients:
\begin{gather}
|\widehat{f}(n)|\le a'_f\exp[-a_f\cdot n^{(1+\beta)/(2+\beta)}],\qquad n\ge 0.\label{fhatn}
\end{gather}
with $\widehat f(n)$ from \eqref{eq:taylor}. We consider this class because in our application to the Jacobi matrices we would like to concentrate on the situations which are close to those  considered by Pavlov and far away from those considered by Na{\u\i}mark. 

This class coincides with the Carleman--Gevrey class 
\begin{multline*}
C_A\{(n!)^{(2+\beta)/(1+\beta)}\}(\mathbb T)\\=
\bigl\{f\in C^\infty_A(\mathbb D):|f^{(n)}(z)|\le b_f^{n+1}(n!)^{(2+\beta)/(1+\beta)},\, n\ge 0,\, z\in\mathbb D\bigr\}.
\end{multline*}

By a theorem of Evsey Dyn'kin, the class $\mathfrak A_\beta$ coincides with the class $\mathfrak C_\beta$ of the planar Cauchy transforms 
of functions $\varphi$ with support in $\mathbb D(2)\setminus \mathbb D$ such that 
\begin{gather}
| \varphi(z)|\le d'_f \rho_\beta(d_f(|z|-1)),\qquad 1<|z|<2,\label{xde4}\\
\rho_\beta(x)=\exp\Bigl(-\frac1{x^{1+\beta}}\Bigr),\qquad x>0\notag,
\end{gather}
with $d_f,d'_f$ depending on $a_f,a'_f$ and $\beta$. For more details, see \cite{Dy1} and Section~\ref{extension}.

It is known (and it follows from the divergence of the corresponding sum \eqref{pnsum}) that the classes $\mathfrak A_\beta$ and $\mathfrak C_\beta$ are analytic quasianalytic. 

In this paper we get an upper bound on the number of zeros of $f$ from such classes in the closed unit disc,  
$N_f={\rm card}\,(Z_f\cap\overline{\mathbb D})$, normalized by the condition $|f(0)|\ge \exp(-A)$, 
in terms of $A$ and $\beta$. 


We formulate our main theorem first for the special case $a_f\asymp1$,  $a'_f\asymp1$ in \eqref{fhatn}, where the statement is somewhat clearer.

\begin{theorem}
\label{main1}
Let $f$ be  in $\mathfrak A_\beta$ with $a_f\asymp1$, $a'_f\asymp1$ 
or, equivalently, in $\mathfrak C_\beta$  with $d_f\asymp1$, $d'_f\asymp1$  
and let $|f(0)|\ge \exp(-A)$ for some $A\ge 1$. 
\begin{enumerate}
\item[(a)] If $\beta=0$, then $N_f\le \exp(c\sqrt{A})$ with some absolute constant $c$.
\item [(b)] If $0<\beta\leq\beta_0$, then
\begin{equation*}
N_f\le 
\begin{cases}
\exp(c\sqrt{A}),\quad & A\le \beta^{-2},\\
A^{(2/\beta)+1}\beta^{(4/\beta)+2}\exp\frac{c}{\beta},\quad & \beta^{-2}\le A\le \beta^{-4},\\
A^{(1/\beta)+1}\exp\frac{c}{\beta},\quad & A\ge \beta^{-4},
\end{cases} 
\end{equation*}
for some absolute constant $c$ depending only on $\beta_0$.
\end{enumerate}
\end{theorem}

This upper bound has the interesting feature of revealing a certain phase transition. We will also show the (almost) sharpness of our bound for $\beta=0$ in Section \ref{sharp}.

To formulate the main theorem  for the general case we denote
\begin{align*}
d&=a_f^{-\frac{2+\beta}{1+\beta}},
\\
A'&= A+\log(a'_fd^2).
\end{align*}

\begin{theorem}
\label{main}
Let $f$ be in $\mathfrak A_\beta$ and let $|f(0)|\ge \exp(-A)$ with $A'\ge 1$. 
\begin{enumerate}
\item[(a)] If $\beta=0$, then $N_f\le  \frac{c}d \min\{A'd,1\} \exp(c\sqrt{A'd})$
for some absolute constant $c$. 
\item [(b)] If $0<\beta\leq\beta_0$, then
\begin{align*}
&N_f\le \min\biggl(c A'\bigl(1+(A' d^{1+\beta})^{\frac 1\beta} \bigr),\\
&\qquad\qquad\begin{cases}
c d^{-(1+\beta)}\, \exp(c\sqrt{A'd}),\quad & A'\le d^{-1}\beta^{-2},\\
e^{c/\beta} d^{-(1+\beta)} \max\Bigl((A'd^{1+\beta} \beta^2)^{\frac{2+\beta}\beta},1 \Bigl),\quad & A'\ge d^{-1}\beta^{-2}
\end{cases} \biggr)
\end{align*}
for some positive $c$ depending only on $\beta_0$. 
\end{enumerate}
\end{theorem}

The proof of Theorem \ref{main} will be given in Subsection \ref{sec:proofmain}.




\subsection[Non-selfadjoint Jacobi matrices]{Non-selfadjoint Jacobi matrices}
\label{Jintro}

Our main application of \newline\noindent Theorem \ref{main} is estimating from above the number of eigenvalues of discrete non-symmetric Schr\"odinger operators, and, more generally, non-symmetric complex Jacobi matrices.

We now formulate our results precisely. We consider Jacobi matrices of the form
$$
J = \begin{pmatrix}
b_0 & c_0 & 0 & \dots & \dots \\
a_0 & b_1 & c_1 & 0 & \dots \\
0 & a_1 & b_2 & c_2 &  \dots\\
 \dots& 0 & a_2 & b_3 & \dots  \\
 \dots&  \dots& \dots & \dots & \dots
\end{pmatrix}
$$
with complex sequences $(a_n)$, $(b_n)$ and $(c_n)$ satisfying the conditions
$$
\lim_{n\to\infty} a_n = \lim_{n\to\infty} c_n = \frac12 \,,
\qquad \lim_{n\to\infty} b_n = 0 \,.
$$

We consider $J$ as an operator in $\ell^2(\mathbb N_0)$. The above conditions on the coefficients imply that the essential spectrum of $J$ is $[-1,1]$ and therefore the spectrum of 
$J$ in $\bC\setminus [-1,1]$ consists of isolated points which are eigenvalues of finite algebraic multiplicity.

Let us assume that the sequences $(a_n-1/2)_{n\ge 0}$, $(b_n)_{n\ge 0}$, $(c_n-1/2)_{n\ge 0}$ are in $\ell^1$. Under this assumption, 
$J-J_0$ is trace class (where $J_0$ is the matrix with $b_n=0$ and $a_n=c_n=1/2$ for all $n$), and therefore 
the perturbation determinant (see \cite{D1,D2,D3})
$$
\Delta(z) := \det \Bigl( \bigl(J-(z+z^{-1})/2\bigr)\bigl(J_0 - (z+z^{-1})/2\bigr)^{-1} \Bigr)
$$
is well-defined. It is known (see, for instance, \cite{GoEg}) that this function is analytic in the unit disc, that
$\Delta(0) = 1$, that for any $z$ with $|z|<1$ one has $\Delta(z)=0$ if and only if $(z+z^{-1})/2$ is an eigenvalue of $J$,  
and that the order of the zero coincides with the algebraic multiplicity of the corresponding eigenvalue.


To reduce our spectral problem to one in complex analysis, we are first interested in the coefficients in the power series expansion of the determinant $\Delta$ at the origin. We write
$$
\Delta(z) = \sum_{j=0}^\infty \delta_j z^j, 
$$
with $\delta_0=1$. 
The following proposition shows that certain bounds on the coefficients of the Jacobi matrix lead to bounds on the Taylor coefficients of $\Delta$.

\begin{prop}\label{jostbound}
Assume that for some $B,D>0$ and $1/2\leq \gamma\leq 1$,
$$
|2b_n| + |4 a_n c_n -1| \le D e^{-B n^\gamma}, \qquad n\ge 0.
\label{X12}
$$
Then 
$$
|\delta_j| \le  D_1\exp( -(B/4) j^\gamma),
\qquad  j\ge1\,,
$$
where $D_1=C_1D(1+B^{-1/\gamma})\exp\bigl(C_2D(1+B^{-2/\gamma})\bigr)$ and $C_1,C_2$ are absolute constants.
\end{prop}


The proof of this proposition is given in Section~\ref{coeff}. 

Let $N_J$ be the number of eigenvalues of $J$ in $\mathbb C\setminus[-1,1]$, where eigenvalues are counted with their algebraic multiplicity. 
Given $B$ and $D$ we denote
\begin{align*}
d&=B^{-1/\gamma},\\
A'&=D(1+B^{-2/\gamma})+\log(D(1+B^{-1/\gamma})B^{-2/\gamma}).
\end{align*}

\begin{theorem}\label{mainso}
Let $1/2\le \gamma\leq \gamma_0<1$, and let $J$ be a Jacobi matrix such that
$$ 
|2b_n| + |4 a_n c_n -1| \le D e^{-B n^\gamma}, \qquad n\ge 0,
$$
for some $B,D> 0$. Assume that $A'\geq 1$.
\begin{enumerate}
\item[(a)] If $\gamma=1/2$, then $N_J\le  \frac{c}d\exp(c\sqrt{A'd})$
for some absolute positive constant $c$. 
\item [(b)] If $1/2<\gamma\leq\gamma_0$, then
\begin{align*}
&N_J\le \min\biggl(c A'\bigl(1+(A'^{1-\gamma} d^{\gamma})^{1/(2\gamma-1)} \bigr),\\
&\begin{cases}
c d^{-\frac\gamma{1-\gamma}} \exp(c\sqrt{A'd}),\ & A'\le d^{-1}(2\gamma-1)^{-2},\\
e^{\frac c{2\gamma-1}} d^{-\frac\gamma{1-\gamma}} \max\Bigl((A'd^{\frac\gamma{1-\gamma}} (2\gamma-1)^2)^{\frac 1{2\gamma-1}},1\Bigl),\ & A'\ge d^{-1}(2\gamma-1)^{-2}
\end{cases} \biggr)
\end{align*}
for some positive $c$ depending only on $\gamma_0$. 
\end{enumerate}
\end{theorem}

This theorem follows immediately from Theorem \ref{main}, applied to $f=\Delta$ and $\gamma=(1+\beta)/(2+\beta)$, and taking into account the bounds from Proposition \ref{jostbound}. (More precisely, the constant $A'$ provided by Theorem \ref{main} differs from the constant $A'$ above by some absolute constants depending only on $\beta$. The fact that these constants can be omitted follows as in the proof of Theorem \ref{main}.)

It is worth singling out the following special case where $B=1$. This gives a bound on the growth of the number of eigenvalues in the strong coupling limit.

\begin{corollary}\label{mainsocor}
Let $1/2\le \gamma<1$, and let $J$ be a Jacobi matrix such that
$$ 
|2b_n| + |4 a_n c_n -1| \le D e^{-n^\gamma}, \qquad n\ge 0,
$$
for some $D\geq 1$.
\begin{enumerate}
\item[(a)] If $\gamma=1/2$, then $N_J\le  \exp(c\sqrt{D})$
for some absolute positive constant $c$.
\item [(b)] If $1/2<\gamma<1$, then
$
N_J\le c_\gamma D^{\gamma/(2\gamma-1)}$
for a constant $c_\gamma$ depending only on $\gamma$. 
\end{enumerate}
\end{corollary}

For comparison purposes we note that if $|2b_n| + |4 a_n c_n -1| \le D e^{-n}$, $n\ge 0$, then a simple application of Jensen's inequality to $\Delta$ gives the bound $N_J \leq c D$.


\subsection{Plan of the paper}

In Section~\ref{cruder} we give our first estimate on the number of zeros in analytic quasianalytic classes which works for $\beta$ away from $0$. 
Another estimate using a propagation of smallness technique and demonstrating a phase transition is given in Section~\ref{localization}. 
Theorem~\ref{main} (and its special case, Theorem~\ref{main1}) follow from Theorems~\ref{t1}, \ref{t2}, and \ref{thext_small_af}. 
Section~\ref{sharp} is devoted to the sharpness of our estimate in the case $\beta=0$.  
The proof of Proposition~\ref{jostbound} (a Jost type estimate) is contained in Section~\ref{coeff}.
Finally, in Section~\ref{extension}, for the sake of completeness, we give a variant of Dyn'kin construction to establish the equality  
$\mathfrak A_\beta= \mathfrak C_\beta$, $\beta\ge 0$. 
\medskip


\section{First estimate for \text{$\beta>0$}}
\label{cruder}

In this section we present our first method of estimating the number of zeros of functions in analytic quasianalytic Carleman classes 
$\mathfrak C_\beta$. It works only for $\beta>0$, and for a large set of parameters $\beta,A$ it gives results weaker than that
in Section~\ref{localization}. In particular, it does not allow to see the {\it phase transition} of Theorem \ref{main}, when $\beta$ becomes very small with respect to $A$. On the other hand, this method is somewhat simpler than that in Section~\ref{localization}.

Let $0<\beta\le \beta_0$. Suppose that $f\in \mathfrak C_\beta$ with $d_f=d$, $d'_f=1$, that is,
$$
|\bar{\pd} f(z)| \le \rho_\beta(d(|z|-1))\,,
$$
and that $|f(0)|\ge \exp(-A)$ for some $A\ge 1$.

We first note that $f$ is bounded by a univeral constant,
\begin{equation}
\label{eq:fbounded}
|f(z)| \leq 2\sqrt{3} \,,
\qquad z\in \mathbb D(2) \,.
\end{equation}
To see this, we write, using Green's formula,
$$
f(z) =\frac1{\pi} \int_{\mathbb D(2)\setminus \mathbb D}\frac{\bar\pd f(\zeta)}{z-\zeta} \,dm_2(\zeta)\,.
$$
Thus,
\begin{align*}
|f(z)| & \leq \frac1{\pi} \int_{\mathbb D(2)\setminus \mathbb D}\frac{\rho_\beta(d_f(|\zeta|-1))}{|z-\zeta|} \,dm_2(\zeta) \leq \frac1{\pi} \int_{\mathbb D(2)\setminus \mathbb D}\frac{1}{|z-\zeta|} \,dm_2(\zeta) \\
& \leq \frac1{\pi} \int_{\mathbb D(\sqrt 3)}\frac{1}{|\zeta|} \,dm_2(\zeta) = 2\sqrt 3 \,,
\end{align*}
where we used a simple rearrangement inequality and the fact that $\mathbb D(\sqrt 3)$ has the same area as $\mathbb D(2)\setminus \mathbb D$.

\subsection{$\bar\pd$-balayage}
\label{pd}
Consider the closed set
$$
K:=\bigl\{z\in\overline{\mathbb D(2)} \setminus \mathbb D:  |f(z)| \le \rho_\beta (d(|z|-1))\bigr\}\,.
$$

Let $0<\eps\le 1$ and let $\Omega$ be the connected component of $\mathbb D(1+\eps) \setminus K$ containing the origin.

We wish to make $f$ analytic in $\Omega$ by only slightly correcting it. To this end we introduce
\begin{equation}
\label{Feg}
F:= f e^g
\end{equation}
in such a way that
$$
\bar\pd g = -\frac{\bar \pd f}{f}
$$
on $\Omega$. Here we have $\bar\pd g =0$ if $\bar \pd f=0$ (in particular on the whole unit disc).

Notice that on $\Omega$ we have by definition $|\frac{\bar \pd f}{f}|\le 1$. So we can find a solution $g$ by the formula
$$
g(z)= -\frac1{\pi} \int_{\Omega} \frac{\bar \pd f(\zeta)}{f(\zeta)}\frac1{z-\zeta} \,dm_2(\zeta)\,.
$$
Then we have
\begin{align*}
|g(z)| & \leq \frac1{\pi} \int_{\Omega} \frac1{|z-\zeta|} \,dm_2(\zeta) \leq \frac1{\pi} \int_{\mathbb D(2)} \frac1{|z-\zeta|} \,dm_2(\zeta) \\
& \leq \frac1{\pi} \int_{\mathbb D(2)} \frac1{|\zeta|} \,dm_2(\zeta) = 4 
\end{align*}
and, hence,
$$
e^{-4} \le |e^g| \le e^{4}\,,
$$
and 
\begin{equation}
\label{F}
e^{-4} |f|\le |F|\le e^{4}|f|\,
\end{equation}
on $\mathbb D(1+\eps)$.

From now on we work only with $F$. It satisfies \eqref{F}, is analytic in $\Omega\supset \bD$, and has exactly the same zeros as $f$ in 
$\overline{\mathbb D} $, see \eqref{Feg}.
Let us list them:
$$
z_1,\dots, z_N\,,
$$
with $N=N_f$.

\subsection{Harmonic measure on $\Omega$}
\label{hm}

Without loss of generality we can assume that $\Omega$ is regular for the Dirichlet problem. Otherwise, just extend slightly $K$ (diminish $\Omega$) and all the rest will work.

Consider the following function $u_\Omega$ on $\Omega$,
$$
u_\Omega = \log|F| +\sum_{k=1}^N G_\Omega (z_k, \cdot)\,,
$$
where $G_\Omega$ is the Green's function for $\Omega$.
It is harmonic on $\Omega$ since the logarithmic singularities of the first term in the right-hand side are compensated by the second one. It is bounded from above by 
a uniform constant on $\partial\Omega$, and, hence, on $\Omega$.
Applying the mean value theorem in $\Omega$, we obtain that 
\begin{equation}
\label{eq:meanvalue1}
\int_{\pd\Omega} u_\Omega(\zeta) \,d \om_\Omega(\zeta) =  \log|F(0)| +\sum_{k=1}^N G_\Omega (z_k, 0)\,,
\end{equation}
where $\om_\Omega$ denotes the harmonic measure on $\Omega$, evaluated at the point $0$. Furthermore, 
$$
\sum_{k=1}^N G_\Omega (z_k, \zeta)=0,\qquad \zeta\in \pd\Omega\,.
$$
Therefore, by \eqref{eq:meanvalue1},
$$
\int_{\pd\Omega} \log |F(\zeta)| \,d \om_\Omega(\zeta)= \log|F(0)| +\sum_{k=1}^N G_\Omega (z_k, 0)\,,
$$
and
\begin{equation}
\label{1F}
\int_{\mathbb D(1+\eps)\cap \pd\Omega } \log \frac{1}{|F(\zeta)|} \,d \om_{\Omega }(\zeta)+\sum_{k=1}^N G_{\Omega } (z_k, 0) \le 
C+A\,
\end{equation}
for some absolute constant $C$.


Indeed, by \eqref{eq:fbounded} and \eqref{F}, the function $|F|$ is bounded from above by an absolute constant and therefore the integral of $\log|F(\zeta)|$ 
over $\pd\Omega \setminus \mathbb D(1+\eps)$ can be majorized by some absolute constant $C$. 

We have 
$$
|F(\zeta)|\le e^4 e^{-1/(d(|\zeta|-1))^{1+\beta}}, \qquad \zeta\in \mathbb D(1+\eps)\cap \pd\Omega .
$$
Therefore, \eqref{1F} gives us that 
\begin{equation}
\label{1F4}
\int_{\mathbb D(1+\eps)\cap \pd\Omega }\frac{d\om_{\Omega }(\zeta)}{(|\zeta|-1)^{1+\beta} } \le d^{1+\beta}(4+C+A)
\le C_1\, d^{1+\beta} A \,,
\end{equation}
and
\begin{equation}
\label{X15}
\sum_{k=1}^N G_{\Omega } (z_k, 0) \le 4+C+A\le C_1\,A\,,
\end{equation}
for some absolute constant $C_1$. (Here we used $A\geq 1$.)

These estimates will be important to complete the proof. However, first we need to establish some simple estimates on the Green's function in 
$\mathbb D(1+\eps)$ and $\Omega $.

\subsection{Green's function of $\Omega $ and conclusion}
\label{Green}

Let us write yet another mean value theorem.
\begin{multline}
\label{G1}
G_{\Omega }(z, 0) \\= G_{\mathbb D(1+\eps)}(z, 0) - \int_{\pd\Omega \setminus\partial \mathbb D(1+\eps)} G_{\mathbb D(1+\eps)}(z,\zeta) 
\,d\om_{\Omega }(\zeta),\qquad z\in\Omega \,.
\end{multline}
In fact, the function $w\mapsto G_{\mathbb D(1+\eps)}(z, w) - G_{\Omega }(z, w)$ is harmonic in $\Omega $ and has the boundary values $G_{\mathbb D(1+\eps)}(z,\zeta) $, $\zeta\in \pd\Omega $; furthermore, $G_{\mathbb D(1+\eps)}(z,\zeta)=0$, $\zeta\in\partial \mathbb D(1+\eps)$.

We will now estimate the first term on the right side of \eqref{G1} from below and show that it is larger than the second term.
Since
$$
G_{\mathbb D(1+\eps)}(z, \zeta)=\log\Bigl| \frac{(1+\eps)-z\bar \zeta/(1+\eps)}{z-\zeta}\Bigr|,
$$
we have 
\begin{equation}
\label{X16}
G_{\mathbb D(1+\eps)}(z, 0) = \log\frac{1+\eps}{|z|} \ge \log(1+\eps) \ge \frac\eps2,\qquad z\in\mathbb D.
\end{equation}
We claim that
\begin{equation}
\label{G2}
G_{\mathbb D(1+\eps)}(z,\zeta)  \le \log\frac{2\eps}{|\zeta|-1},\qquad z\in\mathbb D,\,\zeta\in  \mathbb D(1+\eps)\setminus \mathbb D.
\end{equation}
Indeed, let $s\in[0,\eps)$ and $\zeta\in  \partial\mathbb D(1+s)$, then, by the maximum principle,
\begin{gather*}
\sup_{z\in\mathbb D}G_{\mathbb D(1+\eps)}(z,\zeta)=\sup_{z\in\mathbb T}\log\Bigl| \frac{(1+\eps)-z\bar \zeta/(1+\eps)}{1-z\bar \zeta}\Bigr| 
\\= \sup_{|w|=s+1}\log\Bigl| \frac{(1+\eps)-w/(1+\eps)}{1-w}\Bigr| .
\end{gather*}
We compute
\begin{align*}
\Bigl| \frac{(1+\eps)-w/(1+\eps)}{1-w}\Bigr|^2
& = \frac{1}{(1+\eps)^2} \frac{(1+\eps)^4 - 2(1+\eps)^2\Re w + |w|^2}{1- 2\Re w + |w|^2} \\
& = 1+\frac{1}{(1+\eps)^2}  \frac{ \left((1+\eps)^2- |w|^2\right) \left( (1+\eps)^2 -1 \right)}{1- 2\Re w + |w|^2} .
\end{align*}
Among $w\in \partial\mathbb D(1+s)$, this is clearly maximized at $w=1+s$, and therefore
$$
\sup_{z\in\mathbb D}G_{\mathbb D(1+\eps)}(z,\zeta) = \log \frac{(1+\eps)-(1+s)/(1+\eps)}{s}
$$
Bounding $(1+\eps)-(1+s)/(1+\eps)\le 2\eps$ for $s<\eps$ we obtain \eqref{G2}.

By \eqref{1F4} and \eqref{G2} we obtain that 
\begin{multline*}
\int_{\mathbb D(1+\eps)\cap \pd\Omega } G_{\mathbb D(1+\eps)}(z, \zeta) \,d\om_{\Omega }(\zeta) \\ \le 
C_1Ad^{1+\beta}\sup_{0<t< \eps}t^{1+\beta}\log\frac{2\eps}{t}\le C_2 A d^{1+\beta}\eps^{1+\beta},\qquad z\in\mathbb D,
\end{multline*}
where $C_2$ depends only on $\beta_0$.

Now we fix 
$$
\eps=\min\Bigl\{1,\frac1{4C_2 A d^{1+\beta}}\Bigr\}^{1/\beta}.
$$ 
and obtain 
that
$$
\int_{\mathbb D(1+\eps)\cap \pd\Omega } G_{\mathbb D(1+\eps)}(z, \zeta) \,d\om_{\Omega }(\zeta) \le \frac\eps4.
$$

Combining this estimate with \eqref{X16} we now obtain from \eqref{G1} that
$$
G_{\Omega }(z, 0)\ge \frac\eps4,\qquad z\in\mathbb D.
$$

By \eqref{X15} we conclude that $N_f\le  4C_1A\eps^{-1}$. Thus we have

\begin{theorem}\label{t1} Let $0<\beta\le \beta_0$, $f\in \mathfrak C_\beta$, $d'_f=1$, $|f(0)|\ge \exp(-A)$ for some $A\ge 1$. 
Then for some positive number $c$ depending only on $\beta_0$ we have 
\begin{equation}
\label{largeA_eq}
N_f\le   c A\bigl(1+A^{1/\beta}d_f^{(1+\beta)/\beta}\bigr).
\end{equation}
\end{theorem}

If $A d_f^{1+\beta} \le 1$, then the estimate \eqref{largeA_eq} is optimal.
On the other hand, if $A d_f^{1+\beta}>1$, then this estimate becomes bad when $\beta\to 0$. To improve it we use 
a more complicated argument in the next section.

\begin{remark}\label{remb0}
The same proof shows that there are positive numbers $c_1,c_2$ such that if $f\in \mathfrak C_0$, $d'_f=1$, $|f(0)|\ge \exp(-A)$ for some $A\ge 1$ and $A d_f\leq c_1$ then 
\begin{equation}
\label{largeA_eq0}
N_f\le   c_2 A \,.
\end{equation}
(Indeed, we can choose $\epsilon=1$ in the above proof and $c_1 = 1/(4C_2)$.) In the next section we will also prove a bound valid without restriction on $Ad_f$, but for small $Ad_f$ the above bound is better.
\end{remark}

\section{Propagation of smallness estimates} 
\label{localization}

Let $0\le \beta\le \beta_0$. Suppose that $f$ is in $\mathfrak C_\beta$ with $d_f=d$, $d'_f=1$, and that $|f(0)|\ge \exp(-A)$ for some $A\ge 1$. 
In this section, we are going to get an upper bound on the number $N=N_f$ of zeros of $f$ in $\overline{\mathbb D}$, 
in terms of $A,d$, and $\beta$, using a propagation of smallness argument applied earlier in a similar way in an analytic 
non-quasianalytic situation in \cite{AA93}. 

\subsection{Imposing additional assumptions}

In the following we will suppose that
\begin{equation}
N\ge C'dA^{(2+\beta)/(1+\beta)},
\label{och1}
\end{equation}
for a large positive number $C'$ depending only on $\beta_0$
and set 
\begin{equation}
M=\lfloor (N/d)^{(1+\beta)/(2+\beta)} \rfloor,
\label{och2}
\end{equation}
where $\lfloor x\rfloor$ is the integer part of a real number $x$. Note that assumption \eqref{och1}, $A\geq 1$ and $C'\geq 2$ imply that $N/d \geq 2$ and therefore
\begin{equation}
\label{eq:comparisonm}
\frac12 (N/d)^{(1+\beta)/(2+\beta)} \leq M \leq (N/d)^{(1+\beta)/(2+\beta)}
\end{equation}

Furthermore, assume that 
\begin{equation}
C'd\ge C_3M^{-1/(1+\beta)},
\label{ech3}
\end{equation}
where $C_3$ is a large positive number to be chosen later on, depending only on $\beta_0$. In particular, we choose $C_3\geq C'$ and then $dM^{1/(1+\beta)}\ge 1$.

Our main arguments will require the additional assumptions \eqref{och1} and \eqref{ech3}. Before presenting them, however, we derive some simple bounds when these assumptions fail. Indeed, if \eqref{och1} is still in place, but \eqref{ech3} fails, then
$$
M < \left( \frac{C_3}{C'} \right)^{1+\beta} d^{-(1+\beta)}
$$
and by \eqref{eq:comparisonm} (which uses \eqref{och1}) we see
\begin{equation}
\label{eq:extra1}
N \leq 2^{(2+\beta)/(1+\beta)} \left( \frac{C_3}{C'} \right)^{2+\beta} d^{-(1+\beta)} \,.
\end{equation}
On the other hand, if \eqref{och1} fails, then
\begin{equation}
\label{eq:extra2}
N < C' d A^{(2+\beta)/(1+\beta)} \,.
\end{equation}


\subsection{Beginning of the argument}

From now on, our arguments use the assumptions \eqref{och1} and \eqref{ech3}.

It will also be convenient to assume that
\begin{equation}
\label{eq:normalizationf}
|f|\leq 1
\qquad \text{on}\ \mathbb D(2) \,.
\end{equation}
In view of \eqref{eq:fbounded} this can be achieved by dividing $f$ by a universal constant (in fact, by $2\sqrt 3$). This does not alter $d=d_f$ and we may still assume that $d_f'$. On the other hand, $A$ is replaced by $A + \ln (2\sqrt 3)$. Since $A\geq 1$, we have $A+\ln(2\sqrt 3)\leq (1+\ln(2\sqrt 3))A$ and therefore the replacement of $A$ only affects the constants, but does not affect the form of our bounds. Therefore, in the following without loss of generality we assume \eqref{eq:normalizationf}.

We now apply Jensen's formula twice. A first application gives immediately
\begin{equation}
\label{eq:jensenbound}
\frac{1}{2\pi}\int_0^{2\pi} \log |f(e^{it})|\,dt \geq -A \,.
\end{equation}
We next claim that it also gives
\begin{equation}
\label{inside}
{\rm card}\,\bigl(Z_f\cap \mathbb D(1-d^{-1}M^{-1/(1+\beta)})\bigr)\le N/2.
\end{equation}

Indeed, let $z_1, \dots , z_L$ be all zeros such that  $|z_i| <  1-d^{-1}M^{-1/(1+\beta)}$. Then by Jensen's formula and \eqref{eq:normalizationf} we have
$$
\frac{L}{dM^{1/(1+\beta)}} \le \sum_{i=1}^L \log\frac{1}{|z_i|} \le A \le  \Big(\frac{N}{C'd}\Big)^{\frac{1+\beta}{2+\beta}}\,.
$$
Hence, by \eqref{och1}, \eqref{och2} the following holds:
\begin{multline*}
L \le  \Big(\frac{N}{C'd}\Big)^{\frac{1+\beta}{2+\beta}}M^{1/(1+\beta)}d \\ \le
\Big(\frac{N}{C'd}\Big)^{\frac{1+\beta}{2+\beta}}\Big(\frac{N}{d}\Big)^{\frac{1}{2+\beta}}d = (C')^{-\frac{1+\beta}{2+\beta}}  \frac{N}{d} \cdot d \le \frac{N}{2}\,,
\end{multline*}
for $C'\geq 4$. This proves \eqref{inside}.

Next, we choose $\theta\in[0,2\pi]$ such that there are (at least) $M$ zeros of $f$ in 
$$
\Omega_{\theta,M}=\Bigl\{re^{i\phi}:1-d^{-1}M^{-1/(1+\beta)}\le r\le 1,\,\theta\le \phi\le \theta+d^{-1}M^{-1/(1+\beta)}\Bigr\}.
$$
Denote the first $M$ zeros by $v_j$, $1\le j\le M$. Rotating the disc, we can assume that 
$\theta=0$.

Since
$$
f(z)=\frac1\pi \int_{\mathbb D(2)\setminus \mathbb D} \frac{h(w)}{z-w}\,dm_2(w),\qquad z\in \mathbb D(2),
$$
with
$$
|h(w)|\le \rho_\beta(d(|w|-1)),\qquad w\in \mathbb D(2)\setminus \mathbb D,
$$
for every zero $v$ of $f$ in $\overline{\mathbb D}$ we have
$$
f(z)=\frac{v-z}\pi\int_{\mathbb D(2)\setminus \mathbb D} \frac{h(w)}{(z-w)(v-w)}\,dm_2(w),
$$
and then for $1\le K\le M$ we have 
$$
f(z)=\frac{\prod_{j=1}^{K}(v_j-z)}\pi\int_{\mathbb D(2)\setminus \mathbb D} \frac{h(w)}{\prod_{j=1}^{K}(v_j-w)}\frac{1}{z-w}\,dm_2(w).
$$

Furthermore, a rough estimate gives us that 
$$
\frac{|h(w)|}{\prod_{j=1}^{K}|w-v_j|}\le \sup_{0<x<1}\rho_\beta(xd)x^{-K},\qquad w\in \mathbb D(2)\setminus \mathbb D.
$$

Hence, for every  $1\le K\le M$ and  for all $t$ such that $d^{-1}M^{-1/(1+\beta)}\le t\le 1$ the following holds:
\begin{multline*}
|f(e^{it})|\le 4(2t)^K\sup_{0<x<1}\rho_\beta(dx)x^{-K}\le 4(2dt)^K\sup_{x>0}\rho_\beta(x)x^{-K}\\
=4(2dt)^K\exp\Bigl[\frac{K}{1+\beta}\log\frac{K}{e(1+\beta)}\Bigr].
\end{multline*}
Minimising with respect to $K$ we obtain that   
\begin{equation}
|f(e^{it})| \le \exp\bigl(-3(2td)^{-(1+\beta)}\bigr),\qquad d^{-1}M^{-1/(1+\beta)}\le t\le 1.
\label{z1}
\end{equation}
(Note that $(2dt)^K\exp\Bigl[\frac{K}{1+\beta}\log\frac{K}{e(1+\beta)}\Bigr]$ viewed as a function of the real variable $K$ has a unique minimum at $K=(1+\beta)(2dt)^{-(1+\beta)}$. By the assumption $t\geq d^{-1} M^{-1/(1+\beta)}$, this is $\leq 2^{-(1+\beta)} (1+\beta) M< M$.)


\subsection{Reduction. Small values on an interval}

By a fractional linear map, we transform the function $f/10$ into an analytic function $g$ in the lower half-plane 
that extends $C^1$-smoothly to the whole plane 
and satisfies the estimates 
\begin{gather}
|\overline{\partial}g(z)|\le \rho_\beta(\Im z),\qquad |z|\le C_1d,\, \Im z>0,\notag\\
|g(z)|\le \frac12,\qquad |z|\le C_1d,\notag\\
\int_{0}^{C_1d}\log|g(x)|dx\ge -2Ad,\label{tt}
\end{gather}
for some positive absolute constant $C_1$. The first estimate here follows immediately from the corresponding bound on $\overline{\partial}f$, the second one from \eqref{eq:fbounded} and the third one from \eqref{eq:jensenbound}.

Furthermore, \eqref{z1} now reads as 
\begin{equation}
\label{z1g}
|g(x)| \le  e^{- \frac{C_2}{x^{1+\beta}}},\qquad x\in [M^{-1/(1+\beta)}, C_1d],
\end{equation}
for some positive $C_2$ depending only on $\beta_0$.

Set 
\begin{equation}
y_0=C_3M^{-1/(1+\beta)}
\label{ttt}
\end{equation}
with a universal constant $C_3$ to be determined.

By \eqref{z1g}, and \eqref{ech3}, we have 
$$
y_0^\beta\int_0^{y_0}\log|g(x)|dx 
\le -\begin{cases}
C_2\frac{C_3^\beta-1}{\beta},\qquad & \beta>0,\\
C_2\log C_3,\qquad & \beta=0.
\end{cases}
$$

Therefore, given a positive absolute constant $C_4$ to be fixed later on, we can choose $C_3>1$ in such a way that 
\begin{equation}
y_0^\beta\int_0^{y_0}\log|g(x)|dx\le -C_4.
\label{z2}
\end{equation}

\subsection{Propagation of smallness} 
Now we apply an iterative procedure similar to that used in \cite{AA93}.  
Set 
\begin{equation}
I_0=-\int_0^{y_0}\log|g(x)|dx
\label{z32}
\end{equation}
and define two sequences $(\gamma_k)_{k\ge 1}$, $\gamma_k\in (0,1]$, $k\ge 1$, and $(I_k)_{k\ge 1}$ in the following inductive way. For $k\ge 1$ set 
\begin{gather}
\label{14}
-\log\rho_\beta(2^{k-1}y_0\gamma_k)=C_{5}\frac{\gamma_kI_{k-1}}{2^{k-1}y_0},\\
\label{11}
I_k=I_{k-1}-C_{6} 2^{k-1}y_0\log \rho_\beta(2^{k-1}y_0\gamma_k)=(1+C_5C_6\gamma_k)I_{k-1},
\end{gather}
for some small positive absolute constant $C_5$ and for some positive absolute constant $C_6$ to be fixed later on.

Equation \eqref{14} can be rewritten as 
$$
(2^{k-1}y_0\gamma_k)^{-1-\beta}=C_5\frac{\gamma_kI_{k-1}}{2^{k-1}y_0},
$$
or, equivalently, as 
\begin{equation}
\label{12}
\gamma_k^{2+\beta}=\frac1{C_{5}(2^{k-1}y_0)^{\beta}I_{k-1}}.
\end{equation}
Since the numbers $I_k$ increase, to prove that such $\gamma_k\in (0,1]$ exist for every $k\ge 1$, 
we need only to check that 
\begin{equation}
\label{z41}
\frac1{C_{5}y_0^{\beta}I_{0}}\le 1
\end{equation}
which follows from \eqref{z2} with $C_4=1/C_5$.

Let us check by induction that
\begin{equation}
\label{stst}
\int_0^{2^ny_0}\log|g(x)|dx\le -I_n, \qquad 0\le n\le \log_2\frac{C_1d}{2y_0}.
\end{equation}

The base case $n=0$ follows from \eqref{z32}. If \eqref{stst} holds for $n=k-1$, then a simple estimate of the Poisson integral 
together with relation \eqref{14} shows that for some positive absolute constant $C_5\in(0,1]$ we have 
\begin{multline*}
\log|g(z)|\le C_5\frac{2^{k-1}y_0\gamma_k}{(2^{k-1}y_0)^2}\int_0^{2^{k-1}y_0}\log|g(x)|dx\\ \le 
\log\rho_\beta(2^{k-1}y_0\gamma_k),\qquad z\in T_k=[2^{k-1}y_0,2^ky_0]-i2^ky_0\gamma_k.
\end{multline*}
Next we consider the rectangle 
$$
U_k=\bigl\{z\in\mathbb C: -2^ky_0\gamma_k\le \Im z\le 2^{k-1}y_0\gamma_k,\, 0\le \Re z\le 2^{k+1}y_0\bigr\}
$$
and the auxiliary function
$$
g_k(z)=g(z)-\frac1\pi \int_{U_k}\frac{\overline{\partial}g(\zeta)}{z-\zeta}dm_2(\zeta).
$$
It is clear that $g_k$ is analytic on $U_k$ and bounded by $1$. 
Furthermore, for some positive absolute constant $C_7$, 
$$
\log|g_k(z)|\le C_7\log\rho_\beta(2^{k-1}y_0\gamma_k),\qquad z\in T_k\subset \partial U_k.
$$
Since $\gamma_k\le 1$, a simple geometric argument shows that 
$$
\omega(x,T_k,U_k)\ge C_8>0,\qquad x\in J_k=[2^{k-1}y_0,2^ky_0],
$$
and by the theorem on harmonic estimation (see, for example, \cite[Section VII\,B]{Koo})we have 
$$
\log|g_k(x)|\le C_7C_8\log\rho_\beta(2^{k-1}y_0\gamma_k),\qquad x\in J_k,
$$
for some positive absolute constant $C_8$.

Hence,
$$
\log|g(x)|\le C_6\log\rho_\beta(2^{k-1}y_0\gamma_k),\qquad x\in J_k,
$$
for some positive absolute constant $C_6$.
Furthermore,
$$
\int_0^{2^ky_0}\log|g(x)|dx\le -I_{k-1}+C_6\cdot 2^{k-1}y_0\log\rho_\beta(2^{k-1}y_0\gamma_k)=-I_k.
$$
Thus, \eqref{stst} is proved.

\subsection{Estimating the number of zeros. Case $\beta=0$.}
Relations \eqref{11} and \eqref{12} give us that 
$$
I_k=I_{k-1}+\sqrt{C_5}C_6\sqrt{I_{k-1}},\qquad k\ge 0.
$$
Therefore,
$$
I_k\ge C k^2,\qquad  k\ge 1.
$$
and hence, by \eqref{tt}, \eqref{ttt} and \eqref{stst} with $n=\lfloor \log_2 (C_1 d)/(2y_0) \rfloor$ we find 
$$
\left( \lfloor \log_2 (C_1 dM)/(2C_3) \rfloor \right)^2 \leq \frac{2}{C} Ad \,.
$$
Thus, if $(C_1 dM)/(2C_3)\geq 4$, then
$$\lfloor \log_2 (C_1 dM)/(2C_3) \rfloor \geq (1/2) \log_2 (C_1 dM)/(2C_3)
$$
and therefore
$$
( \log (C_1 dM)/(2C_3) )^2 \leq \frac{8 (\log 2)^2}{C} Ad \,.
$$
According to \eqref{eq:comparisonm} this implies
$$
N \leq \frac{16 C_3^2}{C_1^2} \, \frac{1}{d} \, e^{2 \sqrt{8(\log 2)^2/C}\, \sqrt{Ad} } \,.
$$
This is the claimed bound.

On the other hand, if $(C_1 dM)/(2C_3)< 4$, then, again by \eqref{eq:comparisonm},
$$
N < \frac{2^8\, C_3^2}{C_1^2} \, \frac{1}{d} \,.
$$
This bound is, up to universal constants, better than the claimed one.

We now recall that so far, we worked under assumptions \eqref{och1} and \eqref{ech3}. If these fail, then we have the bounds \eqref{eq:extra1} and \eqref{eq:extra2}. We claim that both of these bounds are, up to universal constants, better than the claimed ones. This is clear for \eqref{eq:extra1}. For \eqref{eq:extra2} it follows from the fact that $dA^2 \leq (4/e)^4 \, d^{-1} \, e^{\sqrt{Ad}}$.

We summarize our findings as follows

\begin{theorem}\label{delta01}
Let $f\in \mathfrak C_0$ with $d'_f=1$ and $|f(0)|\ge \exp(-A)$ for some $A\ge 1$. Then for some positive absolute constant $c$ we have 
$$
N_f\le \frac{c}{d_f}\exp(c\sqrt{Ad_f}) \,.
$$
\end{theorem}

\begin{remark}
Taking into account Remark \ref{remb0}, we obtain the slightly stronger bound
$$
N_f\le \frac{c}{d_f} \min\{1, Ad_f\} \exp(c\sqrt{Ad_f}) \,.
$$
\end{remark}


\subsection{Estimating the number of zeros. Case $\beta>0$. Phase transition.}
Set 
$$
R_k=I_k(2^k y_0)^\beta,\qquad k\ge 0.
$$
Relations \eqref{11} and \eqref{12} give us that  
$$
R_k\ge R_{k-1}+C_9R_{k-1}^{(1+\beta)/(2+\beta)}
$$
with $C_9=2^\beta C_5^{(1+\beta)/(2+\beta)}C_6>0$. By \eqref{z41}, $R_0\ge C_5^{-1}$.

As in the case $\beta=0$, we obtain that
$$
R_k\ge C_{10}\, k^{2},\qquad k\ge 1,
$$
for some positive $C_{10}$ depending only on $\beta_0$. 
Hence, 
$$
I_k\ge
\frac1{2^{2\beta_0}} C_{10}k^2y_0^{-\beta},\qquad 1\le k\le \frac{2\beta_0}{\beta}.
$$

By \eqref{tt}, \eqref{ttt}, and \eqref{stst} with $n=\min\{\lfloor \frac{2\beta_0}{\beta} \rfloor, \lfloor \log_2 \frac{C_1 d}{2y_0}\rfloor \}$ we have
$$
\min\left(\lfloor \frac{2\beta_0}{\beta} \rfloor, \lfloor \log_2 \frac{C_1 d}{2y_0}\rfloor \right)^2 M^{\beta/(1+\beta)}
\leq \frac{2^{1+2\beta_0} C_3^\beta}{C_{10}}  Ad
$$
We claim that we may assume that
\begin{equation}
\label{eq:loglarge}
C_1d/(2y_0) = C_1 d M^{1/(1+\beta)}/(2C_3)\geq 4 \,.
\end{equation}
Indeed, if this does not hold, we obtain using \eqref{eq:comparisonm} a bound which is of the same form as \eqref{eq:extra1}, up to possibly a different constant.

Assumption \eqref{eq:loglarge} together with $2\beta_0/\beta\geq 2$ allows one to simplify the previous bound to
\begin{equation}
\label{ech2}
C_{11} \min\Bigl(\frac1\beta,\log (C_{12}dM^{1/(1+\beta)})\Bigr)^2M^{\beta/(1+\beta)} \leq Ad
\end{equation}
for some positive $C_{11},C_{12}$ depending only on $\beta_0$.

We now distinguish two cases, according to the size of $Ad\beta^2$. Suppose first that $Ad\beta^2<1$. Note that, by \eqref{och1} and \eqref{och2}, if $C'$ is chosen sufficiently large, then
\begin{equation}
\label{ech18}
C_{11}M^{\beta/(1+\beta)}\ge 1 \,.
\end{equation}
This together with the assumption $Ad\beta^2<1$ implies that the minimum in \eqref{ech2} is attained at $\log (C_{12}dM^{1/(1+\beta)})$ and therefore \eqref{ech2} becomes
$$
C_{11}\log^2 (C_{12}dM^{1/(1+\beta)})M^{\beta/(1+\beta)} \leq Ad \,.
$$
Using \eqref{ech18} again to bound the left side from below, we conclude by \eqref{och2} that 
\begin{equation}
\label{delta0}
N \le \frac{C}{d^{1+\beta}}\exp(C\sqrt{Ad})
\end{equation}
for some positive $C$ depending only on $\beta_0$.

We recall that this bound was derived under the additional assumptions \eqref{och1} and \eqref{ech3}. If one of the restrictions \eqref{och1} and \eqref{ech3} does not hold, then we have \eqref{eq:extra1} and \eqref{eq:extra2}. The first of these is clearly better than \eqref{delta0}. To prove this also for \eqref{eq:extra2} we use the fact that $A\geq 1$ and therefore
\begin{align*}
d A^{(2+\beta)/(1+\beta)} & = d^{-(1+\beta)} (A d)^{2+\beta} A^{-\beta(2+\beta)/(1+\beta)} \leq d^{-(1+\beta)} (A d)^{2+\beta} \\
& \leq  \left( \frac{2+\beta}{2e} \right)^{(2+\beta)/2}  d^{-(1+\beta)} \exp(\sqrt{A d}).
\end{align*}
Thus, we have shown \eqref{delta0} under the sole assumption that $Ad\beta^2<1$.

Next, we discuss the case $Ad\beta^2\ge 1$. We assume first that in addition $\log (C_{12}dM^{1/(1+\beta)})> 1/\beta$. Then, \eqref{ech2} gives us that 
$$
C_{11} M^{\beta/(1+\beta)} \leq dA\beta^2 \,.
$$
By \eqref{eq:comparisonm}, we obtain
\begin{equation}
N \le  d(Ad\beta^2)^{(2+\beta)/\beta}\exp(C/\beta)
\label{R16a}
\end{equation}
for some positive $C$ depending only on $\beta_0$. On the other hand, if $\log (C_{12}dM^{1/(1+\beta)})\le  1/\beta$, then, by \eqref{eq:comparisonm} we have
\begin{equation}
N \le  d^{-(1+\beta)}\exp(C/\beta)
\label{R16b}
\end{equation}
for some positive $C$ depending only on $\beta_0$.

To get \eqref{R16a} and \eqref{R16b} we still have used conditions \eqref{och1} and \eqref{ech3}. If one of the restrictions \eqref{och1} and \eqref{ech3} does not hold, we conclude from \eqref{eq:extra1} and \eqref{eq:extra2}, still assuming $Ad\beta^2\ge 1$, that 
$$
N \le \max\Bigl(d(Ad\beta^2)^{(2+\beta)/\beta}\exp(C/\beta),Cd^{-(1+\beta)}\exp(C/\beta)\Bigl)
$$
for some positive $C$ depending only on $\beta_0$. Indeed, this is clear for \eqref{eq:extra1}. In order to show that the right side of \eqref{eq:extra2} is smaller than the expressions on the right sides of \eqref{R16a} and \eqref{R16b}, we distinguish according to whether $Ad^{1+\beta} \beta^{2(1+\beta)}\leq 1$ or not.

We summarize our findings in the following theorem. We observe a {\it phase transition} in the estimate of the number of zeros depending on $A$, $d$ and $\beta$.

\begin{theorem}\label{t2} Let $0<\beta\le \beta_0$, $f\in \mathfrak C_\beta$, $d'_f=1$, $|f(0)|\ge \exp(-A)$ for some $A\ge 1$. 
Then for some positive $c$ depending only on $\beta_0$, we have
$$
N_f\le 
\begin{cases}
 \frac{c}{d_f^{1+\beta}}\exp(c\sqrt{Ad_f}),\quad & A\le d_f^{-1}\beta^{-2},\\
\max\Bigl(d_f(Ad_f\beta^2)^{\frac{2+\beta}\beta}e^{\frac{c}\beta},cd_f^{-(1+\beta)}e^{\frac{c}\beta}\Bigl),\quad & A\ge d_f^{-1}\beta^{-2}.
\end{cases} 
$$
\end{theorem}


\subsection{Proof of Theorem \ref{main}}\label{sec:proofmain}

Let $f$ be in $\mathfrak A_\beta$ and assume that $|f(0)|\ge \exp(-A)$. According to Theorem \ref{thext_small_af} we have $f\in \mathfrak C_\beta$ with $d_f = C a_f^{-(2+\beta)/(1+\beta)} = C d$ and $d_f'=C_1 a_f' a_f^{-2(2+\beta)/(1+\beta)} = C_1 a_f' d^2$. Then the function $\tilde f = f/d_f'$ satisfies $d_{\tilde f} = d_f=Cd$, $\tilde d_{\tilde f}' = 1$ and $|\tilde f(0)| \geq e^{-\tilde A}$ with
$$
\tilde A = A + \log (C_1 a_f' d^2) = A' + \log C_1 \,.
$$
We may assume that $C_1\geq 1$ and therefore $\tilde A \geq A'\geq 1$. Therefore, Theorems \ref{t1}, \ref{delta01} and \ref{t2} applied to $\tilde f$ imply the conclusion of Theorem~\ref{main} but with $\tilde A$ in place of $A'$. If we assume, in addition, that $A' \geq \log C_1$, then we have $\tilde A \leq 2A'$ and therefore in the upper bound we can replace $\tilde A$ by $2A'$. On the other hand, if $1\leq A'<\log C_1$, then $\tilde A<2\ln C_1 \leq 2(\log C_1) A'$ and therefore in the upper bound we can replace $\tilde A$ by $2(\log C_1)A'$. This proves the theorem.


\section{Sharpness of the estimate of the number of zeros in the case $\beta=0$}
\label{sharp}

In this section we show that Theorem~\ref{main1} is almost sharp in the case $\beta=0$. 
It looks plausible that the same construction will show the sharpness of Theorem~\ref{main1} for small positive $\beta$. 
We leave this as a separate study to limit the size of the current article. It would also be interesting to understand how sharp is Theorem~\ref{main}. 
That seems to be more delicate task since we need to deal here with three parameters ($A,d,\beta$) and their different influence on the final estimate.

\begin{theorem}\label{tsh} Let $\delta>0$. Given $A\ge A(\delta)$, there exists $f\in \mathfrak C_0$ satisfying \eqref{xde4}  
with some absolute constants $d_f$, $d'_f$  
and such that $|f(0)|\ge \exp(-A)$, $N_f\ge \exp(A^{1/2-\delta})$.
\end{theorem}

\begin{proof}
It will be convenient for us to construct first a function $g$ analytic in the right half-plane $\mathbb C_+$ 
with good estimates on the 
$\bar\partial$-derivative in the left half-plane $\bC_-$ and such that $|g(1)|= \exp(-A)$.

Let $\eps\in(0,1/10)$ be a small number to be chosen later on. Denote by 
$\Pi$ the standard strip $\bigl\{x+iy\in\mathbb C: |y|<\frac{\pi}2\bigr\}$. We set 
$$
h(y)=\min \Bigl(1,\frac{\eps |y|}{\log(1/|y|)}\Bigr),\qquad y\in\mathbb R, 
$$
and consider a domain $\Omega$ given by
$$
\Omega=\bigl\{ x+iy\in\mathbb C: x > -h(y)\bigr\}.
$$
which is slightly bigger than $\bC_+$. 

Let $\chi: \Om \to \bC_+$ be the conformal map fixing the points $0$, $1$, and infinity. 
Furthermore, let $\cB=\{\log z: z \in \Omega\}$. Then we can write
$$
\chi= \exp\circ \,\varphi\circ \log,
$$
where $\varphi$ is a conformal map from the curvilinear strip $\cB$ onto the standard strip $\Pi$ fixing the points $0$, $\pm\infty$. 
The strip $\cB$ at $-\infty$ looks like  
$\bigl\{x+iy\in\mathbb C:|y| <  s(x)\bigr\}$, $s(x)=\frac{\pi}2 + \frac{\eps}{|x|} +O(1/x^2)$, $|s'(x)|=O(1/x^2)$, $x\to-\infty$. 
By Warschawski's distortion theorems \cite{W41} we obtain that 
$$
|\chi(z)|\asymp e^{-\pi \int_0^{\log\frac{1}{|z|}} \frac{dx}{ 2s(x)}}\asymp |z|\Bigl( \log\frac1{|z|}\Bigr)^{\kappa},\qquad z\to 0,
$$
and that
\begin{equation}
\label{Wa1}
|\chi'(z) |\asymp \Bigl( \log\frac1{|z|}\Bigr)^{\kappa},\qquad z\to 0.
\end{equation}
where $\kappa=2\eps/\pi$.

\subsection*{A modified domain}
\label{xA}
Given a small number $x_A>0$, set 
$$
h_*(y)=\max(x_A,h(y))
$$
and consider a domain $\Omega_*$ containing $\Omega$, 
$$
\Omega_*=\bigl\{ x+iy\in\mathbb C: x > -h_*(y)\bigr\}.
$$

Next we consider the outer function $g$ in $\Om_*$ determined by its absolute values on the boundary:  
$$
\log|g(z)| =-\frac {b}{|\Re z|}, \qquad z\in \partial \Om_*,
$$
for some $b=b(\eps)\asymp 1$ to be chosen later on.

Let $h(y_A)=x_A$, $y_A>0$. Set $r_A=(x_A^2+y_A^2)^{1/2}$. 
The boundaries of $\Om$ and $\Om_*$ coincide outside of the disc $\mathbb D(r_A)$; inside the disc $\mathbb D(r_A)$ 
they are different: $\pd\Om\cap \mathbb D(r_A)$ consists of two smooth curves belonging to the set $\bigl\{x+iy\in\mathbb C:x=-h(y)\bigr\}$  
while $\pd\Om_*\cap \mathbb D(r_A)$ is just a vertical interval in $\bC_-$.
Set $\Gamma=\pd\Om\cap\pd\Om_*$. 
Let $\om$ be harmonic measure on $\Om$, evaluated at point $1$.

We want to choose $x_A$ in such a way that 
\begin{equation}
\label{xA1}
-\int_{\pd\Om} \log |g(z)| \,d\om(z) =A.
\end{equation}

Notice that $\om(\pd\Om\setminus \Gamma) \asymp y_A \bigl(\log\frac1{y_A}\bigr)^\kappa$. Hence, 
\begin{equation}
\label{xA2}
-\int_{\pd\Om\setminus\Gamma} \log |g(z)|\, d\om(z)\asymp y_A \Bigl(\log\frac1{y_A}\Bigr)^\kappa\cdot \frac{\log\frac{1}{y_A}}{y_A}= o(A),
\end{equation}
for suitable $x_A$ to be chosen later on. 

Next, let us require that 
\begin{equation}
\label{xA3}
-\int_{\Gamma} \log |g(z)|\, d\om(z) \asymp A.
\end{equation}
This integral is equivalent (see \eqref{Wa1}) to
$$
\int_{y_A}^1 \frac{\log\frac{1}{s}}{s}\Bigl( \log\frac1{s}\Bigr)^\kappa\,ds=\frac1{2+\kappa}\,\Bigl(\log\frac1{y_A}\Bigr)^{2+\kappa}\,.
$$
Finally, we choose $x_A$ by the equality
$$
\int_{y_A}^1 \frac{\log\frac{1}{s}}{s}\Bigl( \log\frac1{s}\Bigr)^\kappa\,ds=A.
$$ 
Then \eqref{xA2} and \eqref{xA3} are true and \eqref{xA1} becomes true if we choose the number $b$ in the definition of $g$ appropriately.

Thus, we have 
$$
\log\frac1{x_A}\asymp A^{\frac12-\tau}\,,
$$
with $\tau=\eps/(2(\pi+\eps))$.

Since $g$ is outer in $\Om$, we have
$$
-\log|g(1)|= -\int_{\pd\Om} \log |g(z)| \,d\om(z) = A.
$$

\subsection*{How smooth is $g$?}
\label{sm}
We claim that $g|\bC_+$ extends to a function $\tilde g$ which is $C^1$-smooth in the whole complex plane,   
\begin{equation}
\label{dbar1}
|\bar\pd \tilde g(z)|\le Ce^{-\frac{C_1}{|\Re z|}},\, -1<\Re z<0,
\end{equation}
for some absolute constants $C,C_1$, and $\tilde g(z)$ vanishes for $\Re z\le -1$.

Indeed, consider a smooth function $\psi$ such that $\psi(x+iy)=1$ on $\{x+iy:x\ge -h_*(y)/2\}$, $\psi(x+iy)=0$ on $\{x+iy:x\le -h_*(y)\}$, and $0\le\psi\le 1$ everywhere. We can find such $\psi$ with 
\begin{equation}
|\bar\pd\psi(x+iy)|\le \frac{C}{h_*(y)},\qquad x+iy\in\mathbb C.
\label{do25}
\end{equation}
Furthermore, an easy estimate of harmonic measure gives us that 
\begin{equation}
\label{do26}
|g(x+iy)| \le C e^{-\frac{C_1}{h_*(y)}},\quad -h_*(y)\le x\le -h_*(y)/2,
\end{equation}
with some absolute constants $C,C_1$. 
Put $\tilde g := \psi\,g$. Now, property \eqref{dbar1} follows from \eqref{do25} and \eqref{do26}.

\subsection*{Imposing zeros}
\label{zeros}
By a linear fractional transformation, we can transfer $g$ to $\bD$ and its extension $\tilde g$  to $\bD(2)$. 
Then $g$ belongs to the class $\mathfrak C_{0}$ and $g(0)=e^{-A}$.  The only problem is that our $g$ is an outer function and so has no zeros whatsoever. On the other hand, $g$ is very small on the arc $I_A$ centered at the point $1\in \bT$ of length $2y_A$.
In fact, 
$$
|g(\zeta)| \le e^{-\frac{C}{x_A}}\le e^{-C_* e^{A^{\frac12-\tau}} },\qquad \zeta\in I_A,
$$   
for some absolute constants $C,C_*>0$. 

For $N\ge 1$ to be chosen later on, let the points $\{x_j\}_{j=1}^N$ divide $I_A$ into $N$ equal arcs of length $\frac{2y_A}{N}$.
Set $\ell(z)=\Pi_{j=1}^N (z-x_j)$, and let $L$ be the Lagrange interpolation polynomial that interpolates the function 
$g$ at the points $\{x_j\}_{j=1}^N$:
$$
L(z)= \ell(z)\sum_{j=1}^N \frac{g(x_j)}{\ell'(x_j)(z-x_j)}\,.
$$

To estimate $|L(0)|$ we use the equalities $|\ell(0)/x_j|=1$, $1\le j\le N$, and a lower bound for $|\ell'(x_j)|$: 
$$
|\ell'(x_j)|\ge (2y_A/N)^N.
$$
Now, 
$$
|L(0)|\le N e^{-C_* e^{A^{\frac12-\tau}} } e^{N\log N} e^{C_1 NA^{\frac12-\tau}},
$$
for some absolute constant $C_1$.

Choose
$$
N= \Bigl\lfloor e^{A^{\frac12-2\tau}}\Bigr\rfloor.
$$
Then $L(0)\le e^{-A}/2$ for $A\ge A(\eps)$. 

Set $f=g-L$. Then $|f(0)|\ge e^{-A}/2$ and $f$ has $N$ zeros in the closed unit disc.

Notice that our argument for estimating $L(0)$ works also for $L(z)$. In the same way we obtain that 
$$
|L(z)|\le  N3^{N} e^{-C_* e^{A^{\frac12-\tau}} } e^{N\log N} e^{C_1NA^{\frac12-\tau}},\qquad |z|=2,
$$
and 
$$
|L(z)|\le e^{-A}/2,\qquad |z|=2,\,A\ge A_1(\eps).
$$
By the maximum principle, we conclude that 
$$
\sup_{\mathbb D(2)}|L|\le e^{-A}/2,\qquad A\ge A_1(\eps).
$$

Now consider a cut-off smooth function $\Psi$ equal to $1$ in $\bD(3/2)$ and zero outside $\bD(2)$ and put
$\tilde f = \tilde g- \Psi\, L$;
This function extends $f$, 
$$
|\bar{\pd}\tilde f(z)| \le Ce^{-\frac{C_1}{|z|-1}},\qquad z\in \bD(2)\setminus \bD,
$$
for some absolute constants $C,C_1>0$, 
and $n_f\gtrsim e^{A^{\frac12-2\tau}}$, $A\ge A_1(\eps)$.
\end{proof}


\section{Application to non-selfadjoint Jacobi matrices}\label{coeff}

\subsection{Proof of Proposition~\ref{jostbound}}

The first ingredient in the proof of Proposition \ref{jostbound} is the following 
result, which estimates the coefficients $\delta_j$ in terms of the numbers $a_n$, $b_n$, and $c_n$. In the self-adjoint case it appears, e.g., in Section 10.1 of \cite{Te}. The same proof works in the non-selfadjoint case, where it appears, e.g., as Theorem 2.3 in \cite{GoEg}. Set 
$$
H(N) := \sum_{n=N}^\infty \bigl( |2b_n| + |4 a_n c_n -1| \bigr).
$$

\begin{lemma}\label{jostabstract}
For every $j\geq 1$,
$$
|\delta_j| \le H([j/2]) \Bigl( \prod_{n=0}^\infty (1+H(n)) \Bigr) \,,
$$
where $[j/2]$ is the integer part of $j/2$.
\end{lemma}

The second ingredient ingredient we need is an elementary bound on exponential sums.

\begin{lemma}\label{expsum}
Let $1/2\le\gamma\le 1$. Then for all $B>0$ and $N\geq 0$ we have
$$
\sum_{n=N}^\infty e^{-B n^\gamma}\le C\Bigl( 1+ B^{-1/\gamma} \Bigl(1+ \bigl(BN^\gamma\bigr)^{(1/\gamma)-1} \Bigr) \Bigr) e^{-B N^\gamma}
$$
with an absolute constant $C$. In particular, for all $B>0$, $N\geq 0$ and $c\in(0,1)$ we have
$$
\sum_{n=N}^\infty e^{-B n^\gamma} \le C_{c}  (1+B^{-1/\gamma})e^{-c B N^\gamma} \,,
$$
where $C_{c}$ depends only on $c$.
\end{lemma}

\begin{proof}[Proof of Lemma \ref{expsum}.]
By monotonicity, we have
\begin{multline*}
\sum_{n=N}^\infty e^{-B n^\gamma} \leq e^{-B N^\gamma} + \int_N^\infty e^{-B x^\gamma}\,dx \\= e^{-B N^\gamma} + \frac1{\gamma B^{1/\gamma}} \int_{BN^\gamma}^\infty e^{-y} y^{(1/\gamma) -1}\,dy \,.
\end{multline*}
We now use the fact that for $0\leq\alpha\leq 1$,
\begin{equation}
\int_Y^\infty e^{-y} y^\alpha \,dy \le 2e (1+ Y^\alpha) e^{-Y} \qquad Y\ge 0.
\label{X13}
\end{equation}
Indeed, we have
\begin{multline*}
\int_Y^\infty e^{-y} y^\alpha \,dy  = e^{-Y} Y^\alpha + \alpha \int_Y^\infty e^{-y} y^{\alpha-1}\,dy \\
\le e^{-Y} Y^\alpha + \alpha Y^{\alpha-1} \int_Y^\infty e^{-y}\,dy = (1+\alpha Y^{-1}) e^{-Y} Y^\alpha\\ \le (1+ Y^\alpha) e^{-Y}, \qquad Y\ge 1.
\end{multline*}
Furthermore, 
$$
\int_Y^\infty e^{-y} y^\alpha \,dy \leq \int_0^\infty e^{-y} y^\alpha \,dy =\Gamma(\alpha+1) \le 2e (1+ Y^\alpha) e^{-Y}, \,\,\,\, 0\le Y\le 1.
$$
Together, these two inequalities prove \eqref{X13}.

Applying \eqref{X13} with $\alpha =(1/\gamma) -1$,  
we obtain
$$
\int_{BN^\gamma}^\infty e^{-y} y^{(1/\gamma)-1}\,dy \le 2e \Bigl(1+ \bigl(BN^\gamma\bigr)^{(1/\gamma)-1} \Bigr) e^{-B N^\gamma}
$$
and therefore
$$
\sum_{n=N}^\infty e^{-B n^\gamma} \le C \Bigl( 1+ B^{-1/\gamma} \Bigl(1+ \bigl(BN^\gamma\bigr)^{(1/\gamma)-1} \Bigr) \Bigr) e^{-B N^\gamma} 
$$
for some absolute constant $C$.
\end{proof}

To complete the proof of Proposition~\ref{jostbound}, 
we combine Lemmas~\ref{jostabstract} and \ref{expsum}. Fix $0<c<1$. By assumption, we have
$$
H(N) \le D \sum_{n=N}^\infty e^{-B n^\gamma},\qquad N\ge 0,
$$
and therefore Lemma~\ref{expsum} implies that
$$
H(N) \le  C_cD(1+B^{-1/\gamma}) e^{-cBN^\gamma}, \qquad N\ge 0.
$$

Moreover, using the estimate $\log (1+x)\leq x$ we obtain that 
$$
\log \Bigl( \prod_{n=0}^\infty (1+H(n)) \Bigr) = \sum_{n=0}^\infty \log (1+H(n)) \leq \sum_{n=0}^\infty H(n) \,.
$$
Therefore, from the bound we have just derived, we get 
$$
\log \Bigl( \prod_{n=0}^\infty (1+H(n)) \Bigr) \le C_cD(1+B^{-1/\gamma}) \sum_{n=0}^\infty e^{-cB n^\gamma} \,.
$$
Applying again Lemma \ref{expsum} we find
$$
\log \Bigl( \prod_{n=0}^\infty (1+H(n)) \Bigr) \le C'_cD(1+B^{-2/\gamma})\,.
$$
In view of Lemma~\ref{jostabstract} these bounds imply the proposition.


\section{Smooth extensions with estimates on $\bar{\pd}$. Dyn'kin construction}
\label{extension}

At the beginning of the 1970-s Dyn'kin proposed a general approach of representing functions in different smoothness classes as traces 
of asymptotically holomorphic functions, that is, functions satisfying some quantitative restrictions on the $\bar\pd$-derivative. 

In particular, it follows from the results in \cite{Dy1} that 
$$
\mathfrak A_\beta=\mathfrak C_\beta,\qquad \beta\ge 0.
$$

Here we give a short proof of (a quantitative version of) the inclusion $\mathfrak A_\beta\subset\mathfrak C_\beta$. The opposite inclusion is not needed in this paper, but we give a proof of it after the proof of the theorem.



\begin{theorem}
\label{thext_small_af}
Let $0\le\beta\le\beta_0$ and let $f$ be analytic in the unit disc and satisfy \eqref{fhatn} with $a'_f= 1$. 
Then $f$ 
extends to a $C^1$-smooth function with compact support in 
$\mathbb D(2)$ (we denote this extension by the same symbol $f$) in such a way that
$$
f(z) =\frac1{\pi} \int_{\mathbb D(2)\setminus \mathbb D}\frac{\bar\pd f(\zeta)}{z-\zeta} \,dm_2(\zeta)\,,
$$
and
\begin{equation}
\label{om}
|\bar\pd f(z)| \le d'_f\rho_\beta( d_f(|z|-1)), \qquad z\in \mathbb D(2)\setminus \mathbb D,
\end{equation}
where
$$
\rho_\beta(x) = e^{-\frac{1}{x^{1+\beta}}} \,,
$$
and
$$
d_f = Ca_f^{-\frac{2+\beta}{1+\beta}} \,,\qquad 
d'_f = C_1a_f^{-2\frac{2+\beta}{1+\beta}}  \,,
$$
with some $C,C_1$ depending only on $\beta_0$.
\end{theorem}

\begin{proof}
Let 
$$
\gamma=\Bigl(\frac{a_f}{2^{2+\beta}}\cdot\frac{1+\beta}{2+\beta}\Bigr)^{2+\beta}.
$$ 
Set
$$
N(0)=0,\quad N(m)=2^{(2+\beta)m}\gamma,\quad m\ge 1.
$$
Consider $S_m= \sum_{k\ge N(m)}^\infty \hat f(k) z^k$, the tail of the Taylor series of $f$.

Let $\varphi_m$ denote the $C^1$ smooth function equal to $0$ on $\mathbb C\setminus \mathbb D(1+2^{-m})$, equal to $1$
on $\mathbb D(1+2^{-m-1})$ and such that $\nabla\varphi_m$ has compact support in $\mathbb D(1+2^{-m})\setminus \overline{\mathbb D(1+2^{-m-1})}$ 
and $|\bar\pd \varphi_m|\le C2^m$, $m\ge 0$.

Now define
\begin{equation}
\label{fext}
f = \sum_{m\ge 0}^\infty \varphi_m \cdot (S_m - S_{m+1}).
\end{equation}
In particular, on the unit circle this sum is just $f=\sum_{m\ge 0} (S_m - S_{m+1} )$. 
Thus, formula \eqref{fext} gives an extension of our original function $f$ to $\mathbb D(2)\setminus \mathbb D$. 
Furthermore, this extended $f$ has compact support in $\mathbb D(2)$. 

Let us estimate the $\bar\pd$-derivative of $f$.
If $z$ belongs to $\mathbb D(1+2^{-m})\setminus \overline{\mathbb D(1+2^{-m-1})}$, then only one $\bar\partial\varphi_k$ 
(namely $\bar\partial\varphi_m$) is not $0$. 
The terms $\varphi_m \cdot (S_m - S_{m+1})$, $k\ne m$, obviously give zero contribution to $\bar\pd f$, 
because  $S_k-S_{k+1}$ are just analytic polynomials.

Thus, if $z\in \mathbb D(1+2^{-m})\setminus \overline{\mathbb D(1+2^{-m-1})}$, then
\begin{gather*}
|\bar\pd f(z)| \le C2^m |S_m(z)-S_{m+1}(z)| \\ \le C2^m \sum_{2^{(2+\beta)m}\gamma\le s<2^{(2+\beta)(m+1)}\gamma} 
e^{-a_fs^{\frac{1+\beta}{2+\beta}}} (1+ 2^{-m})^s\\
\le C\gamma2^{m+(2+\beta)(m+1)} \exp\bigl(-a_f2^{(1+\beta)m}\gamma^{\frac{1+\beta}{2+\beta}}+2^{-m+(2+\beta)(m+1)}\gamma\bigr)
\\=C\gamma2^{(3+\beta)m+2+\beta} \exp\Bigl(- \frac{1}{2+\beta}a_f\gamma^{\frac{1+\beta}{2+\beta}}  2^{(1+\beta)m} \Bigr).   
\end{gather*}
Thus, 
$$
|\bar\pd f(z)| \le u_fe^{ -v_f 2^{(1+\beta)(m+1)}},\quad z\in \mathbb D(1+2^{-m})\setminus \overline{\mathbb D(1+2^{-m-1})},\,m\ge 0,
$$
with 
$$
v_f=Ca_f^{2+\beta} \,,\qquad 
u_f=C_1a_f^{-(4+2\beta)/(1+\beta)} \,,
$$
with some $C,C_1$ depending only on $\beta_0$.
This proves \eqref{om}.

By construction, $f$ has compact support in $\mathbb D(2)$, and hence, Green's formula allows us to restore $f(z)$ as follows:
$$
f(z)=\frac1{\pi} \int_{\mathbb D(2)} \frac{\bar\pd f(\zeta)}{z-\zeta}\,dm_2(\zeta) = \frac1{\pi} \int_{\mathbb D(2)\setminus \mathbb D} \frac{\bar\pd f(\zeta)}{z-\zeta}\,dm_2(\zeta) \,.
$$
We are done.
\end{proof}

\begin{remark}
In the opposite direction, if $0\le\beta\le\beta_0$ and if $f$ is analytic in the unit disc and satisfies \eqref{om}, then it satisfies \eqref{fhatn} with
$$
a_f = C d_f^{-\frac{1+\beta}{2+\beta}} \,,
\qquad
a_f' = C_1 d_f'
$$
with some $C$, $C_1$ depending only on $\beta_0$.
\end{remark}

\begin{proof}
Indeed, in this case, for every $0<\eps<1$ we have 
\begin{multline*}
|\hat f(n)|=\Bigl| \frac1{2\pi}\int_{\partial\bD} f(z)z^{-n-1}\,dz\Bigr|\\=\Bigl| \frac1{2\pi}\int_{\partial\bD(1+\eps)} f(z)z^{-n-1}\,dz-
\frac1{\pi}\int_{\bD(1+\eps)\setminus\bD} \bar\partial f(z)z^{-n-1}\,dm_2(z)\Bigr|\\ \le 
\frac{Cd'_f}{(1+\eps)^n}+2d'_f\rho_\beta( d_f\eps),\qquad n\ge 0.
\end{multline*}
On the other hand, we have 
$$
|\hat f(n)|=\Bigl|  
\frac1{\pi}\int_{\bD(2)\setminus\bD} \bar\partial f(z)z^{-n-1}\,dm_2(z)\Bigr|
\le 
 2d'_f\rho_\beta( 2d_f).
$$

If $nd_f^{1+\beta}>1$, then we set $\eps=n^{-1/(2+\beta)}d_f^{-(1+\beta)/(2+\beta)}<1$ and conclude that 
$$
|\hat f(n)|\le Cd'_f\exp\Bigl(-\frac12d_f^{-(1+\beta)/(2+\beta)}n^{(1+\beta)/(2+\beta)}\Bigr).
$$
Otherwise, if $0<n\le d_f^{-(1+\beta)}$, then 
$$
|\hat f(n)| \le  2d'_f\rho_\beta( 2d_f)\le 
Cd'_f\exp\Bigl(-\frac1{2^{1+\beta}}d_f^{-(1+\beta)/(2+\beta)}n^{(1+\beta)/(2+\beta)}\Bigr)
$$
for some absolute constant $C$. Finally,
$$
|\hat f(0)| \le  Cd'_f .
$$
\end{proof}



\begin{thebibliography}{26}

\bibitem{AA93}  {\sc A.~Borichev}, \textit{Analytical quasianalyticity and asymptotically holomorphic functions}, St.-Petersburg Math.\ Journal
\textbf{4} (1993), no. 2, 259--272.

\bibitem{BoGoKu} {\sc A.~Borichev, L.~Golinskii, S.~Kupin}, \textit{A Blaschke-type condition and its application to complex Jacobi matrices}, 
Bull.\ London Math.\ Soc.\ \textbf{41} (2009), 117--123.

\bibitem{Ca} {\sc L.~Carleson}, \textit{Sets of uniqueness for functions regular in the unit circle}, 
Acta Math.\ \textbf{87} (1952), 325--345.

\bibitem{Da} {\sc E.~B.~Davies}, \textit{Non-self-adjoint differential operators}, Bull.\ London Math.\ Soc.\ \textbf{34} (2002), no. 5, 513--532. 

\bibitem{DeHaKa} {\sc M.~Demuth, M.~Hansmann, G.~Katriel}, \textit{On the discrete spectrum of non-selfadjoint operators}, J. Funct.\ Anal.\ \textbf{257} (2009), no. 9, 2742--2759.

\bibitem{Dy1}  {\sc E.~M.~Dyn'kin}, \textit{Functions with a prescribed bound for $\partial f/\partial \bar z$, and a theorem of N.~Levinson}, 
in Russian, Mat.\ Sb.\ {\bf 89} (1972), no. 2, 182--190; English translation in Math.\ USSR-Sb.\ {\bf 18} (1972), no. 2, 181--189.

\bibitem{EmTr} {\sc M.~Embree, L.~N.~Trefethen}, \textit{Pseudospectra gateway},\qquad\qquad\qquad\qquad\phantom{A}
http://www.comlab.ox.ac.uk/pseudospectra.

\bibitem{FrLaLiSe} {\sc R.~L.~Frank, A.~Laptev, E.~H.~Lieb, R.~Seiringer}, \textit{Lieb--Thirring inequalities for Schr\"odinger operators with complex-valued potentials}, Lett.\ Math.\ Phys.\ \textbf{77} (2006), 309--316.

\bibitem{FrLaSa} {\sc R.~L.~Frank, A.~Laptev, O.~Safronov}, \textit{On the number of eigenvalues of Schr\"odinger operators with complex potentials}, 
J. Lond.\ Math.\ Soc.\ (2) \textbf{94} (2016), no. 2, 377--390.

\bibitem{FrSa} {\sc R.~L.~Frank, J.~Sabin}, \textit{Restriction theorems for orthonormal functions, Strichartz inequalities and uniform Sobolev estimates},  Amer.\ J. Math.\ \textbf{139} (2017), no. 6, 1649--1691.

\bibitem{Fr3} {\sc R.~L.~Frank}, \textit{Eigenvalue bounds for Schr\"odinger operators with complex potentials. III}, Trans.\ Amer.\ Math.\ Soc.\ 
\textbf{370} (2018), no. 1, 219--240.

\bibitem{FrSi} {\sc R.~L. Frank, B.~Simon}, \textit{Eigenvalue bounds for Schr\"odinger operators with complex potential}, Journal of Spectral Theory \textbf{7} (2015), no. 3, 633--658.

\bibitem{D1} {\sc I.~Gohberg, M.~Krein}, \textit{Introduction to the theory of linear nonselfadjoint operators}, Translations of Mathematical Monographs 18, AMS, Providence, 1969.

\bibitem{GoEg} {\sc L.~B.~Golinski\u\i, I.~E.~Egorova}, \textit{On limit sets for the discrete spectrum of complex Jacobi matrices}, 
in Russian, Mat.\ Sb.\ {\bf 196} (2005), no. 6, 43--70; English translation in Sb.\ Math.\ {\bf 196} (2005), no. 5--6, 817--844.

\bibitem{GoKu} {\sc L.~Golinskii, S.~Kupin}, \textit{Lieb--Thirring bounds for complex Jacobi matrices}, Lett.\ Math.\ Phys.\ \textbf{82} (2007), no. 1, 79--90.

\bibitem{Ha} {\sc M.~Hansmann}, \textit{An eigenvalue estimate and its application to non-self-adjoint Jacobi and Schr\"odinger operators}, Lett.\ Math.\  Phys.\ \textbf{98} (2011), no. 1, 79--95.

\bibitem{HaKa} {\sc M.~Hansmann, G.~Katriel}, \textit{Inequalities for the eigenvalues of non-selfadjoint Jacobi operators}, Complex Analysis and Operator Theory \textbf{5} (2011), Issue 1, pp. 197--218.

\bibitem{D3} {\sc R.~Killip, B.~Simon}, \textit{Sum rules for Jacobi matrices and their applications to spectral theory}, Ann.\ of Math.\ {\bf 58} (2003), 253--321.

\bibitem{Koo} {\sc P.~Koosis}, \textit{The logarithmic integral. I.}, Cambridge Studies in Advanced Mathematics, 12. Cambridge University Press, Cambridge, 1988.

\bibitem{Ko} {\sc B.~Korenblum}, \textit{Quasianalytic classes of functions in a circle}, Soviet Math. Dokl.\ {\bf 6} (1965), 1155--1158.

\bibitem{LS} {\sc A.~Laptev, O.~Safronov}, \textit{Eigenvalue estimates for Schr\"odinger operators with complex potentials}, Comm.\ Math.\ Phys.\ 
{\bf 292} (2009), 29--54.

\bibitem{Na} {\sc M.~A.~Na{\u\i}mark}, \textit{Investigation of the spectrum and the expansion in eigenfunctions of a nonselfadjoint operator of the second order on a semi-axis}, in Russian, Trudy Moskov.\ Mat.\ Ob\v s\v c.\ \textbf{3} (1954), 181--270.

\bibitem{BSP} {\sc B.~S.~Pavlov}, \textit{On a non-selfadjoint Schr\"odinger operator}, Probl.\ Math.\ Phys., No. I, Spectral Theory and Wave Processes, pp. 102--132, Izdat.\ Leningrad.\ Univ., Leningrad, 1966.

\bibitem{Pa1} {\sc B.~S.~Pavlov}, \textit{On a non-selfadjoint Schr\"odinger operator}, in Russian, 1966 Probl.\ Math.\ Phys., No. 1, Spectral Theory and Wave Processes, pp. 102--132, Izdat.\ Leningrad.\ Univ., Leningrad.

\bibitem{Pa2} {\sc B.~S.~Pavlov}, \textit{On a non-selfadjoint Schr\"odinger operator. II}, in Russian, 1967 Probl.\ Math.\ Phys, No. 2, Spectral Theory, Diffraction Problems, pp. 133--157, Izdat.\ Leningrad.\ Univ., Leningrad.

\bibitem{ReSi4} {\sc M.~Reed, B.~Simon}, \textit{Methods of modern mathematical physics. IV. Analysis of operators}, Academic Press [Harcourt Brace Jovanovich, Publishers], New York--London, 1978.

\bibitem{D2} {\sc B.~Simon}, \textit{Trace ideals and their applications}, Mathematical Surveys and Monographs 120, AMS, Providence, RI, 2005.

\bibitem{Te} {\sc G.~Teschl}, \textit{Jacobi operators and completely integrable nonlinear lattices}, Mathematical Surveys and Monographs \textbf{72}. American Mathematical Society, Providence, RI, 2000.

\bibitem{W41} {\sc S.~E.~Warschawski}, \textit{On conformal mapping of infinite strips}, Trans.\ Amer.\ Math.\ Soc.\ {\bf 51} (1942), 280--335.

\end{thebibliography}
\end{document}